\documentclass[a4paper]{amsart}
\usepackage{hyperref}   
\usepackage{amssymb,xcolor}
\usepackage{stmaryrd} 
\usepackage{enumerate}

\newcommand{\unfoldedcomment}[7]{}
\newcommand{\foldedcomment}[7]{}

\usepackage{mathtools}      
\mathtoolsset{showonlyrefs}

\newtheorem{theorem}{Theorem}[section]
\newtheorem{definition}[theorem]{Definition}
\newtheorem{lemma}[theorem]{Lemma}
\newtheorem{remark}[theorem]{Remark}
\newtheorem{proposition}[theorem]{Proposition}
\numberwithin{equation}{section}
\numberwithin{figure}{section}
\begin{document}

\title[Spectral PMT for asymptotically hyperbolic 3-manifolds]{Spectral positive mass theorem for asymptotically hyperbolic 3-manifolds with
toroidal infinity}


\author{Xiaoxiang Chai}
\address{School of Mathematics and Statistics,
  Key Laboratory of Nonlinear Analysis \& Applications (Ministry of Education), Central China Normal
University, Wuhan 430079, China}
\email{xxchai@kias.re.kr}

\author{Yimin Chen}
\address{School of Mathematics and Statistics, Hainan University, Haikou, China}
\email{yiminchen@hainanu.edu.cn}

\author{Juncheol Pyo}
\address{Department of Mathematics and Institute of Mathematical Science, Pusan National University, Busan, 46241, South Korea \& Korea and School of Mathematics, Korea Institute for Advanced Study, Seoul, 02455, South Korea}
\email{jcpyo@pusan.ac.kr}



\begin{abstract}
We define a mass invariant adapted to the spectral scalar curvature for asymptotically hyperbolic 3-manifolds with toroidal infinity and show its positivity under a lower bound on the spectral scalar curvature. In addition, we show a rigidity theorem and some band width estimates under similar assumptions.
\end{abstract}

\maketitle

\section{Introduction}

A Riemannian manifold $M$ is called \textit{asymptotically flat} if there is a compact set $K\subset M$, such that $M\setminus K$ is diffeomorphic to $\Bbb R^n\setminus B_{1}(0)$ and in the standard coordinates in $\Bbb R^n$ and the metric is close to the flat metric near the infinity. The positive mass theorem (see \cite{schoen1979proof,schoen1981proof,witten1981new}) establishes the non-negativity of a geometric quantity called the ADM mass under the assumption of nonnegative scalar curvature. The case of vanishing mass characterizes the Euclidean space.
%
%
%
%
%

One can also study a manifold close to the hyperbolic space near the infinity, and such a manifold is called an \textit{asymptotically hyperbolic manifold}. See \cite{wang2001mass,chrusciel2003mass,andersson-rigidity-2008} and also an earlier work \cite{minoo-scalar-1989}. 
Instead, a negative scalar curvature lower bound given by the model hyperbolic space is assumed.
Recently, we have seen great progress in the positive mass theorems, in particular the resolution of the higher dimensional positive mass theorem, see \cite{bi2026proof}, \cite{hirsch2026hyperboloidalspacetimepositivemass}, \cite{brendle2026dimensiondescentschemepositive}.


Another model metric of interest in the field of the positive mass theorem is the \( n \)-dimensional
hyperbolic cusp 
\begin{equation} \mathrm{d} t^2 + e^{2t} g_{\mathbb{T}^{n-1}} \label{horocyclic model}
\end{equation}
which by a coordinate change \( r = e^t \) is isometric to
\begin{equation} r^{-2}\mathrm{d} r^2 + r^2 g_{\mathbb{T}^{n-1}}. \label{conformal form}
\end{equation}
In this paper, we are interested in asymptotically hyperbolic 3-manifolds with infinity given by \eqref{horocyclic model} (i.e., \( t, r \to \infty \)) and we give the following definition using the metric \eqref{conformal form}.
\begin{definition}\label{ah non}
  A manifold $(M, g)$ is called asymptotically hyperbolic with toroidal infinity if there is a
  compact set $K$ whose complement is diffeomorphic to a cylinder with torus
  cross-section and in the coordinates given by the diffeomorphism $\psi : (1,
  \infty)\times \mathbb T^2 \to M\backslash K$ \ the metric satisfies
  \begin{equation}
    \psi^{\ast} g = r^{- 2} \mathrm{d} r^2 + r^2 (g_{\mathbb{T}^2} +
  r^{- \kappa} m) + Q_g \label{asymptotics g}
  \end{equation}
  where $\kappa >0$, $r \in (1, \infty)$ is the radial coordinate, $m$ is a symmetric
  2-tensor on $\mathbb{T}^2$ and $Q_g$ satisfies the asymptotics
  \[ |Q_g |  + r | \nabla Q_g | + r^2 | \nabla^2 Q_g | = o (r^{-\kappa}) . \]
\end{definition}

It is worth mentioning that
such manifolds belong to a more general class of ALH manifolds (see \cite{chrusciel-mass-2018}, \cite{huang-scalar-2022}) and the toroidal infinity can also be called a cuspidal end.

Now we recall the mass for such manifolds.

\begin{definition}[\cite{chrusciel-trautman-bondi-2004, chrusciel2003mass}]
  Let $\kappa=1$. If $r (R_g + 6) \in L^1 (M)$, then
  \[ E = \tfrac{1}{|\mathbb{T}^2 |} \int_{\mathbb{T}^2}
     3\ensuremath{\operatorname{tr}}_{g_{\mathbb{T}^2}} m \mathrm{d}
     A_{g_{\mathbb{T}^2}} \]
  is well-defined and called the total
energy, and
  $\operatorname{tr}_{g_{\mathbb{T}^2}} m$ is called the
  {mass aspect function}.
\end{definition}

The positive mass theorem for this mass was shown in
\cite{chrusciel-mass-2018}, see also \cite{alaee-positive-2022} for the
3-dimensional spacetime setting. For the rigidity statement, see \cite{huang-scalar-2022}.
%

Our main goal is to pursue a spectral positive mass theorem of an asymptotically hyperbolic $3$-manifold with toroidal infinity. Such a positive mass theorem deals with the spectral scalar curvature, the spectral mean curvature and requires the notion of a spectral mass. First, we give the following spectral analog of the scalar curvature
and the mean curvature.

\begin{definition}
Given a Riemannian manifold \( (M,g) \) (possibly with boundary), a positive number \( \gamma>0 \)
and a positive function $u\in C^2(\bar {M})$, the quantity
\begin{equation} 
  -\gamma u^{-1}\Delta_g u + \tfrac{1}{2} R_g \label{scc def}
\end{equation}
is called the \( (\gamma,u)\)-spectral scalar curvature. For a hypersurface $\Sigma$ in \(  M \), \[ \gamma u^{-1} u_{\nu} + H \] is called the \( (\gamma,u) \)-spectral mean curvature of $\Sigma$. We often omit the references to \( (\gamma,u) \) for convenience.
We adopt the convention that the normal of a boundary component of a manifold point to the outside of the manifold; let \( \Sigma \) be a boundary component of \( M \) and \( \nu \) be its normal, and
  the mean curvature of $\Sigma$ is given by $H=\operatorname{div}_\Sigma\nu$.
  \end{definition}

  There are two variations
  \begin{align}
    S_{c,\gamma} = & -\gamma u^{-1}\Delta_g u + \tfrac{1}{2} R_g + c \gamma u^{-2} |\nabla u|^2  \label{grad pert} ,\\
    P_{ \alpha, \gamma} = & R_g  - 2 \gamma \Delta_g f + \alpha |\nabla f | ^2 \label {p form sc}
    \end{align}
    of the spectral scalar curvature \eqref{scc def}. For the definition of \( P_{\alpha,\gamma} \), see for instance \cite{deng-curvature-dimension-2021}.
    Evidently, letting \( f = \log u \) in \eqref{p form sc} gives the form \eqref{grad pert} with suitable values of \( c \). A notable example of \eqref{p form sc} is the Perelman's weighted scalar curvature
    \begin{equation}
      P=R_g + 2 \Delta_g f -|\nabla f |^2 . \label{perelman sc}
      \end{equation}
    The form \eqref{grad pert} is actually equivalent to \eqref{scc def} by observing that
    \[  
     - u^{-1}\Delta_g u  + c u^{-2} |\nabla u|^2  = - \tfrac{1}{1-c} u^{-(1-c)} \Delta_g u^{1-c} .
   \]
   In particular, setting \( f = - 2 \log u \) gives \( P = 2 (-2u^{-1} \Delta_g u + \tfrac{1}{2} R_g) \).
   





Now we introduce the following spectral mass.

\begin{definition}\label{m g u}
 We fix $0 < \gamma < 2$ and $\kappa = \tfrac{6 - 2\gamma}{2 - \gamma}$. 
 Let $(M,g)$ be given in Definition \ref{ah non}
and a function $u\in C^{2}(\bar{M})$ which satisfies
  \begin{equation} u = r^{\tfrac{1}{2 - \gamma}} (1 + r^{- \kappa} \zeta + o
     (r^{- \kappa})). \label{u asymptotics} \end{equation}
If \( u^2 (- \gamma u^{-1}\Delta_gu+\tfrac{1}{2}R_g) \in L^1 (M) \),
  the quantity
  \[ E:= E (M,g,\gamma,u) = \tfrac{4\kappa}{2-\gamma} \int_{T^2} (\tfrac{1}{2} \ensuremath{\operatorname{tr}}_{g_{\mathbb{T}^2}}m + \gamma
    \zeta) \sqrt{g_{\mathbb{T}^2}} dA\]
  is called the spectral mass (or energy)
  and the quantity \(\tfrac{1}{2}\ensuremath{\operatorname{tr}}_{g_{\mathbb{T}^2}}m + \gamma
    \zeta \) on $T^2$ is called the spectral mass aspect function.
\end{definition}

Our definition is related to the Perelman's weighted scalar curvatrue \eqref{perelman sc} and the weighted mass in the asymptotically flat setting (see \cite{baldauf-spinors-2022}, cf. \cite{brendle2026dimensiondescentschemepositive}). We can also define the spectral mass with more general asymptotics than \eqref{asymptotics g}, see Chai-Liu-Sun (\textit{work in prepration}).

Our main result is given in the following which is a positive mass theorem for \( E \) assuming suitable bounds for the spectral scalar curvature and the spectral mean curvature of the boundary.

 \begin{theorem} \label{spec pmt}
   Let $(M,g)$ be asymptotically hyperbolic with toroidal infinity as in Definition \ref{ah non} with $K$ being empty, and 
 \( u \) be given by Definition \ref{m g u}.
   If the spectral scalar curvature satisfies
   \[ -\gamma u^{-1}\Delta_g  u + \tfrac{1}{2} R_{g} \geq    - \tfrac{(3-\gamma)(4-\gamma)}{(2-\gamma)^{2}} \]
   and the spectral mean curvature of $\partial M$ satisfies
   \[ \gamma u^{-1}u_{\nu} + H_{\partial M} \geq    -\tfrac{4-\gamma}{2-\gamma},\]
    then the spectral mass \( E(M,
  g,\gamma, u)\) is non-negative. The spectral mass vanishes if and only if $(M, g)$ is
  isometric to
  \[ ([r_0, \infty) \times \mathbb{T}^2, r^{- 2} \mathrm{d} r^2 + r^2
     g_{\mathbb{T}^2}) \]
  for some $r_0>0$ and $u$ is equal to
  $r^{\tfrac{1}{2 - \gamma}}$.
\end{theorem}
\begin{remark}
We can also allow additional boundaries as in \cite[Theorem 1.1]{alaee-positive-2022}.
It is easy to see this is a generalization of the non-spectral positive mass theorem (i.e., the case \( \gamma =0 \)).
\end{remark}

Our approach to Theorem \ref{spec pmt} is based on the spacetime harmonic functions introduced in \cite{hirsch-spacetime-2022} which are tied to three dimensions. It is an interesting question to extend Theorem \ref{spec pmt} to higher dimensions.

The application of spacetime harmonic functions also yields some other geometric results under similar assumptions on the spectral scalar curvature. Using similar techniques, we prove a spectral scalar curvature rigidity theorem for toroidal bands or cusps (Theorem \ref{spectral acg}) and some band width estimates (Theorem \ref{band width estimate}), which were obtained earlier by Chai-Sun \cite{chai-rigidity-2026} using the warped 
$\mu$-bubble method. For the band width estimates in the case of a positive spectral scalar curvature bound, see \cite{hirsch-spectral-2024}.
%
%
%
 For related scalar curvature rigidity, see \cite{minoo-scalar-1989}, {\cite{witten-connectedness-1999}}, {\cite{andersson-rigidity-2008}}.

The spectral scalar curvature rigidity is given as follows.

\begin{theorem}[{cf. \cite[Theorem 1.8]{chai-rigidity-2026}}]
  \label{spectral acg}Let \( M = [- 1, 1] \times \mathbb{T}^2 \) with a metric
  \( g \) and \( u \) be a positive function such that
  \begin{equation}
    \label{scc1} - \gamma u^{- 1} \Delta_g u + \tfrac{1}{2} R_g \geq    \Lambda
  \end{equation}
  where \( 0 \leq   \gamma < 3 \) and \( \Lambda \leq   0 \). Let \(
  \nu_{\pm} \) be the unit normal of the boundary \( \partial_{\pm} M =\{\pm
  1\} \times \mathbb{T}^2 \) pointing to the direction pointing to the outside of
  \(M\). Suppose
  \[ H_{\partial_- M} + \gamma u^{- 1} u_{\nu_-} \geq   -\sqrt{-
     \Lambda \tfrac{4 - \gamma}{3 - \gamma}}  \text{ and } H_{\partial_+ M} +
   \gamma u^{- 1} u_{\nu_+} \geq    \sqrt{- \Lambda \tfrac{4 - \gamma}{3 - \gamma}} 
   . \]
  Then \( (M, g) \) is isometric to
  \[ ([- 1, 1] \times \mathbb{T}^2, g = dt^2 + e^{2 \alpha t}
     g_{\mathbb{T}^2}), \text{ } \alpha = \tfrac{(2 - \gamma) \sqrt{-
     \Lambda}}{\sqrt{(3 - \gamma) (4 - \gamma)}}, \]
  and \( u \) is a constant multiple of \( e^{\beta t} \) where \( \beta =
  \sqrt{\tfrac{- \Lambda}{(3 - \gamma) (4 - \gamma)}} \).
\end{theorem}

\begin{remark}\label{same metric}
It should be observed that when \( 0< \gamma <2 \),
the metrics in Theorem \ref{spectral acg} and Theorem \ref{spec pmt} are in fact the same, the 3-dimensional hyperbolic cusp
\( \mathrm{d}t^2 + e^{2t} g_{\mathbb{T}^2} \) and \( u \) is \( e^{t / (2-\gamma)} \).
The difference only come from by a coordinate change and a scaling, see \eqref{horocyclic model} and \eqref{conformal form}. We also call Theorem \ref{spectral acg} a spectral cuspidal rigidity.

\end{remark}

Now we turn to the band width estimates, which is a research direction pioneered by Gromov \cite{gromov-metric-2018}. He established that if a toroidal band with positive scalar curvature has bounded width which is defined to be the distance
between two boundaries, see also \cite{cecchini-scalar-2024}, \cite{rade-scalar-2023}, \cite{hirsch-rigid-2025}.

To facilitate the description of the band width estimates, we introduce some
notations. Let $\Gamma$ and $\Lambda$ be two constants and we are concerned with ODE
\begin{equation}
  \Gamma \eta^2 + \eta' + \Lambda = 0 \label{eq general ode}
\end{equation}
such that the solution $\eta$ satisfies $\eta' < 0$. To ensure $\eta' < 0$, at
least one of $\Gamma$ and $\Lambda$ should be positive. Indeed the solution to
\eqref{eq general ode} is given by the following
\[ \eta (t) := \eta_{\Lambda, \Gamma} (t) := \left\{ \begin{array}{lc}
     \sqrt{- \Lambda / \Gamma} \coth \left( \sqrt{- \Lambda / \Gamma} t
     \right),     & \Gamma > 0, \Lambda < 0;\\
     \frac{1}{\Gamma t},     & \Gamma > 0, \Lambda = 0;\\
     \sqrt{\Lambda / \Gamma} \cot \left( \sqrt{\Lambda / \Gamma} t \right),    & \Gamma > 0, \Lambda > 0.
   \end{array} \right. \]
Evidently, $\eta_{\Lambda, \Gamma}$ is only well defined on the interval
$I_{\Lambda, \Gamma}$ given by
\[ I_{\Lambda, \Gamma} = \left\{ \begin{array}{cc}
     (0, \infty),   & \Gamma > 0, \Lambda \leq   0;\\
     \left( 0, \pi / \sqrt{\Lambda \Gamma} \right),   & \Gamma > 0, \Lambda
     > 0.
   \end{array} \right. \]

\begin{theorem}[{cf. \cite[Theorem 1.12]{chai-rigidity-2026}}]
  \label{band width estimate}
  Let $\Lambda$ be a constant, $0 \leq   \gamma < 3$ and $\Gamma =
  \frac{3 - \gamma}{4 - \gamma}$. Assume that at least one of $\Gamma$ and
  $\Lambda$ is positive. Let $M=[-1,1]\times\mathbb T^2$, and $t_- < t_+$ be two numbers such that
  \begin{enumerate}[(a)]
  \item there exists a positive function $u$ with $- \gamma u^{- 1} \Delta_g u
  + \frac{1}{2} R_g \geq    \Lambda$;
  \item $H_{\partial_+ M} + \gamma u^{- 1} u_{\nu_+} \geq    \eta (t_+)$ on
  $\partial_+ M$, $H_{\partial_- M} + \gamma u^{- 1} u_{\nu_-} \geq    - \eta
  (t_-)$ on $\partial_- M$,
  \end{enumerate}
then
\[ \ensuremath{\operatorname{width}} (M, g) \leq    t_+ - t_-. \]
The equality occurs if and only if $(M, g)$ is isometric to the model
\[ ([t_-, t_+] \times T^2, d t^2 + \phi (t)^2 g_{\mathbb{T}^2}) \]
where $\phi (t) = \exp \left( \frac{2 - \gamma}{4 - \gamma} \int^t \eta
\right)$ and $g_{\mathbb{T}^{n - 1}}$ is some flat metric on $\mathbb{T}^{n -
1}$ and $u$ is a constant multiple of $\exp \left( \frac{1}{4 - \gamma} \int^t
\eta \right)$.

\end{theorem}

\

The article is organized as follows:

In Section \ref{shf sec}, we review spacetime harmonic functions and establish a spectral analogue of a fundamental integral inequality.
In Section \ref{sec band}, we apply this integral inequality to prove Theorems \ref{spectral acg} and \ref{band width estimate}.
Finally, in Section \ref{sec pmt}, we prove Theorem \ref{spec pmt}.

\section*{\textbf{Acknowledgments}} 
Yimin Chen was supported by a start-up grant from Hainan University (XJ2600000222). Juncheol Pyo was supported by the National Research Foundation of Korea (RS-2025-23524266).

\section{The spacetime harmonic function and integral inequality}\label{shf sec}
The primary tool employed in this paper is the spacetime harmonic function. A triple $(M^3, g, k)$ consisting of a 3-Riemannian manifolds $(M^3,g)$ and a symmetric 2-tensor $k$ is called an {\it{initial data set}}. Let us denote $\Delta$, $\nabla$ and $\nabla^2$ the Laplacian, Levi-Civita connection and Hessian on $M$, respectively. 
\begin{definition}
  A positive smooth function $u$ on $(M^3, g, k)$ is called a spacetime harmonic function with $k=f g$ for a smooth function $f$ on $M$ if it satisfies
  \begin{equation}
    \Delta u + 3 f | \nabla u | = 0. \label{shf}
  \end{equation}
\end{definition}
We denote by $\overline\nabla^2$ the modified Hessian given by
$$\overline\nabla^2 v=\nabla^2 v+f|\nabla v|g$$
for any $v\in C^2(M)$. Then the equation \eqref{shf} can be simply written as $$\overline\Delta v=\operatorname{tr}\overline\nabla^2 v=0.$$
In \cite{hirsch-spacetime-2022}, an application of the Bochner formula and the Gauss-Bonnet theorem yields the following integral inequality.
\begin{proposition}[{\cite{hirsch-spacetime-2022}}]
  \label{hkk formula}
  Let $(M, \partial_{\pm}M, g)$ be a 3-dimensional Riemannian band, and let $f\in C^{\infty}(M)$. Let $v$ be a spacetime harmonic function with $v = \pm c$ on
  $\partial_{\pm} M$, respectively. Then, we have the following integral inequality
      \begin{align}
          & \int_{\partial_- M} 2 | \nabla v | (2 f - H_{\partial_- M}) d A -
    \int_{\partial_+ M} 2 | \nabla v | (2 f + H_{\partial_+ M}) d A\\
     \geq    & \int_{M^3} \left( \frac{| \bar{\nabla}^2 v |^2}{| \nabla v |} +
    (R_g + 6 f^2) | \nabla v | - 4 \langle \nabla f, \nabla v \rangle \right)
    - \int_{c_-}^{c_+} 4 \pi \chi (\Sigma_t) d t
      \end{align}
  where $\bar{\nabla}^2 v = \nabla^2 v + f | \nabla v | g$.
\end{proposition}

In order to deal with the condition on spectral curvature condition, we need a spectral analog of Proposition \ref{hkk formula}.

\begin{proposition}
  \label{spectral hkk f}For the function \( v \) which satisfies
  \[ \Delta_g v + 3 f | \nabla v| = 0 \]
  with \( v = c_{\pm} \) on \( \partial_{\pm} M \) the following inequality holds 
  \begin{align}
    & \int_{\partial_- M} | \nabla v| \left( \tfrac{3 (4 - \gamma)}{3 -
    \gamma} f - 2 (H_{\partial_- M} + \gamma u^{- 1} u_{\nu_-}) \right) 
    \nonumber\\
    & \qquad - \int_{\partial_+ M} | \nabla v| \left( \tfrac{3 (4 -
    \gamma)}{3 - \gamma} f + 2 (H_{\partial_+ M} + \gamma u^{- 1} u_{\nu_+})
    \right)\nonumber \\
    \geq    & (6 - 2 \gamma) \int_M \left| \nabla | \nabla v|^{\tfrac{1}{2}}
    + \tfrac{3}{6 - 2 \gamma} f| \nabla v|^{- \tfrac{1}{2}} \nabla v \right|^2
    - \int_{c_-}^{c_+} 4 \pi \chi (\Sigma_t) \mathrm{d} t \nonumber\\
    & \quad + \int_M (2 (- \gamma u^{- 1} \Delta_g u + \tfrac{1}{2} R_g) +
    \tfrac{9 (4 - \gamma)}{2 (3 - \gamma)} f^2 - \tfrac{3 (4 - \gamma)}{3 -
    \gamma} \langle \nabla f, \tfrac{\nabla v}{| \nabla v|} \rangle) | \nabla
    v| . 
  \end{align}
\end{proposition}

\begin{proof}
  We give a proof for \( v \) with its gradient
non-vanishing everywhere. For the general case, we can apply similar arguments
for \( \varphi = \sqrt{| \nabla v|^2 + \varepsilon^2} \), \( \varepsilon > 0
\) and then take limits, see {\cite{hirsch-spacetime-2022}}.

First, we note that the divergence theorem and that \( v \) is spacetime
harmonic yields
\begin{align}
  & - \tfrac{\gamma}{3 - \gamma} \int_{\partial M} f \nabla_{\nu} v
  \\
  = & - \tfrac{\gamma}{3 - \gamma} \int_M \ensuremath{\operatorname{div}}_g (f
  \nabla v) \\
  = & - \tfrac{\gamma}{3 - \gamma} \int_M (\langle \nabla f, \nabla v \rangle
  + f \Delta_g v) \\
  = & - \tfrac{\gamma}{3 - \gamma} \int_M (\langle \nabla f, \nabla v \rangle
  - 3 f^2 | \nabla v|) . 
\end{align}
Since \( v = c_{\pm} \) on \( \partial_{\pm} M \) and \( c_+ > c_- \), \(
\nabla_{\nu} v = \pm | \nabla v| \) on \( \partial_{\pm} M \). Hence,
\begin{equation}
  - \tfrac{\gamma}{3 - \gamma} \left( \int_{\partial_+ M} f| \nabla v| -
  \int_{\partial_- M} f| \nabla v| \right) = - \tfrac{\gamma}{3 - \gamma}
  \int_M (\langle \nabla f, \nabla v \rangle - 3 f^2 | \nabla v|) .
  \label{exra div term}
\end{equation}
We add \eqref{exra div term} to the integral inequality in Proposition
\ref{hkk formula}, we obtain that
\begin{align}
  & \int_{\partial_- M} | \nabla v| \left( \tfrac{3 (4 - \gamma)}{3 - \gamma}
  f - 2 H_{\partial_- M} \right) - \int_{\partial_+ M} | \nabla v| \left(
  \tfrac{3 (4 - \gamma)}{3 - \gamma} f + 2 H_{\partial_+ M} \right) 
  \\
  \geq    & \int_M \left( \frac{| \bar{\nabla}^2 v|^2}{| \nabla v|} + (R_g +
  \tfrac{3 (6 - \gamma)}{3 - \gamma} f^2) | \nabla v| - \tfrac{3 (4 -
  \gamma)}{3 - \gamma} \langle \nabla f, \nabla v \rangle \right)\\
  &-
  \int_{c_-}^{c_+} 4 \pi \chi (\Sigma_t) \mathrm{d} t. \label{before kato}
\end{align}
We make use of a Kato type inequality (see {\cite[Remark
4.4]{hirsch-rigid-2025}}) on \( \frac{| \bar{\nabla}^2 v|^2}{| \nabla v|} \)
as
\begin{align}
  & \frac{| \bar{\nabla}^2 v|^2}{| \nabla v|} \\
  \geq    & \frac{3 | \nabla | \nabla v| + f \nabla v|^2}{2 | \nabla v|}
  \\
  = & 6 | \nabla | \nabla v|^{\tfrac{1}{2}} |^2 + 6 \langle \nabla | \nabla
  v|^{\tfrac{1}{2}}, f | \nabla v|^{- \tfrac{1}{2}} \nabla v \rangle +
  \tfrac{3}{2} f^2 | \nabla v| \\
  = & (6 - 2 \gamma) \left| \nabla | \nabla v|^{\tfrac{1}{2}} + \tfrac{3}{6 -
  2 \gamma} f| \nabla v|^{- \tfrac{1}{2}} \nabla v \right|^2 + 2 \gamma |
  \nabla | \nabla v|^{\tfrac{1}{2}} |^2 + (\tfrac{3}{2} - \tfrac{9}{6 - 2
  \gamma}) f^2 | \nabla v| . 
\end{align}
Hence with the above in \eqref{before kato}, we see
\begin{align}
  & \int_{\partial_- M} | \nabla v| \left( \tfrac{3 (4 - \gamma)}{3 - \gamma}
  f - 2 H_{\partial_- M} \right) - \int_{\partial_+ M} | \nabla v| \left(
  \tfrac{3 (4 - \gamma)}{3 - \gamma} f + 2 H_{\partial_+ M} \right) 
  \\
   & \geq\;(6 - 2 \gamma) \int_M \left| \nabla | \nabla v|^{\tfrac{1}{2}} +
  \tfrac{3}{6 - 2 \gamma} f| \nabla v|^{- \tfrac{1}{2}} \nabla v \right|^2 +
  \int_M 2 \gamma | \nabla | \nabla v|^{\tfrac{1}{2}} |^2 \label{after kato}
  \\
  & \quad + \int_M \left( (R_g + \tfrac{9 (4 - \gamma)}{2 (3 - \gamma)} f^2)
  | \nabla v| - \tfrac{3 (4 - \gamma)}{3 - \gamma} \langle \nabla f, \nabla v
  \rangle \right) - \int_{c_-}^{c_+} 4 \pi \chi (\Sigma_t) \mathrm{d} t.\label{determineu}
\end{align}
Now we give an estimate of \( \int_M | \nabla | \nabla v|^{\tfrac{1}{2}} |^2
\) in terms of \( u \): by integration by parts,
\begin{align}
   &- \int_M \varphi^2 u^{- 1} \Delta_g u + \int_{\partial M} \varphi^2 u^{-
  1} u_{\nu} \\
  = & - \int_M \varphi^2 u^{- 1} \Delta_g u + \int_M
 \mbox{div}_g (\varphi^2 u^{- 1} \nabla u) \\
  = & \int_M \langle \nabla (\varphi^2 u^{- 1}), \nabla u \rangle \\
  = & - \int_M | \nabla \varphi - \tfrac{\varphi \nabla u}{u} |^2 + \int_M |
  \nabla \varphi |^2 \leq   \int_M | \nabla \varphi |^2 . 
\end{align}
Letting \( \varphi = | \nabla v|^{\tfrac{1}{2}} \) gives
\[ - \int_M u^{- 1} \Delta_g u | \nabla v| + \int_{\partial M} u^{- 1} u_{\nu}
   | \nabla v| \leq   \int_M | \nabla | \nabla v|^{\tfrac{1}{2}} |^2 . \]
The inequality above and \eqref{after kato} then finish the proof.
\end{proof}

\section{Spectral cuspidal rigidity and the band width estimate}\label{sec band}

In this section, we prove the spectral scalar rigidity for cusps (Theorem \ref{spectral acg}) and the band width estimate (Theorem \ref{band width estimate}). The idea is to select a suitable \( f \) in
Proposition \ref{spectral hkk f} and then perform the rigidity analysis.

Now we give the proof for Theorem \ref{spectral acg}. 

\begin{proof}[Proof of Theorem \ref{spectral acg}]
  Let \( v \) be the function which solves
  \[ \Delta_g v - (6 - 2 \gamma) \sqrt{\tfrac{- \Lambda}{(3 - \gamma) (4 -
     \gamma)}}  | \nabla v| = 0 \]
  with \( v = c_{\pm} \) on \( \partial_{\pm} M \), that is, we set \( f = -
  \tfrac{6 - 2 \gamma}{3} \sqrt{\tfrac{- \Lambda}{(3 - \gamma) (4 - \gamma)}}
  \) in Proposition \ref{spectral hkk f}. This leads to an integral inequality
  for \( v \) as follows
\begin{align}
& 2 \int_{\partial_- M} | \nabla v| \left( - \sqrt{- \Lambda \tfrac{4 -
\gamma}{3 - \gamma}} - (H_{\partial_- M} + \gamma u^{- 1} u_{\nu_-})
\right)  \\
& \qquad - 2 \int_{\partial_+ M} | \nabla v| \left( - \sqrt{- \Lambda
\tfrac{4 - \gamma}{3 - \gamma}} + (H_{\partial_+ M} + \gamma u^{- 1}
u_{\nu_+}) \right) \\
\geq    & (6 - 2 \gamma) \int_M \left| \nabla | \nabla v|^{\tfrac{1}{2}}
- \sqrt{\tfrac{- \Lambda}{(3 - \gamma) (4 - \gamma)}} | \nabla v|^{-
\tfrac{1}{2}} \nabla v \right|^2 - \int_{c_-}^{c_+} 4 \pi \chi (\Sigma_t)
\mathrm{d} t \\
& \quad + \int_M 2 (- \gamma u^{- 1} \Delta_g u + \tfrac{1}{2} R_g -
\Lambda) | \nabla v| .
\end{align}
  It follows that \( \chi (\Sigma_t) \leq    0 \) from the toroidal
  structure of \( M \); then we apply the spectral scalar curvature and
  spectral mean curvature bound, we obtain that
  \begin{equation}
    \nabla | \nabla v|^{\frac{1}{2}} - \sqrt{\tfrac{- \Lambda}{(3 - \gamma) (4
    - \gamma)}} | \nabla v|^{- \frac{1}{2}} \nabla v = 0 \label{gradeq}
  \end{equation}
  almost everywhere. This implies that
  \[ \nabla | \nabla v| - 2 \sqrt{\tfrac{- \Lambda}{(3 - \gamma) (4 -
     \gamma)}} | \nabla v| \nabla v = 0. \]
  Since \( | \nabla v| \neq 0 \) almost everywhere, the metric splits, that
  is, \( g = \frac{dv^2}{| \nabla v|^2} + g_v \), where \( g_v \) is a family
  of metrics on \( \mathbb{T}^2 \).
  
  Firstly from
  \[ d (| \nabla v|^2) = 2| \nabla v|d (| \nabla v|) = \tfrac{4 \sqrt{-
     \Lambda}}{\sqrt{(3 - \gamma) (4 - \gamma)}} | \nabla v|^2 dv, \]
  we have
  \begin{equation}
    d (| \nabla v|^2) \wedge dv = 0, \label{integrable}
  \end{equation}
  and therefore \( d \left( \frac{\nabla v}{| \nabla v|} \right) = 0 \) for \(
  | \nabla v| \neq 0 \), which implies that there exist \( t \), such that \( dt
  = dv / | \nabla v| \), and the metric can be written by
  \[ g = dt^2 + \phi^2 (t) g_0 .\]
  From the argument in \cite{hirsch-spectral-2024},
  \begin{eqnarray*}
    2 \frac{\phi'}{\phi} = H (t) = \frac{\Delta v - \nabla_{tt} v}{| \nabla
    v|} & = & \frac{2 \alpha (\gamma - 3)}{\gamma - 2} + \frac{2
    \alpha}{\gamma - 2} = 2 \alpha,
  \end{eqnarray*}
  therefore we have \( \phi = e^{\alpha t} \). Therefore, the metric is given
  by
  \[ g = dt^2 + e^{2 \alpha t} g_{\mathbb{T}^2} . \]
  Tracing back all the equalities in Proposition \ref{spectral hkk f}, we see
  from \eqref{determineu} that $$\frac{\nabla|\nabla v|^{\frac{1}{2}}}{|\nabla v|^{\frac{1}{2}}}=\frac{\nabla u}{u}$$
  \( u \) must be a constant multiple of \( | \nabla v|^{\tfrac{1}{2}} \),
  and by a direct calculation, it is a constant multiple of \( e^{\beta t} \).
\end{proof}

Now we prove the band width estimate Theorem \ref{band width estimate}.

\begin{proof}
  We prove it by contradiction. Suppose that
  $\ensuremath{\operatorname{width}} (M, g) \geq    t_+ - t_-$. Let $\zeta (x) =
  \min \{ t_+, t_- +\ensuremath{\operatorname{dist}}_g (x, \partial_- M) \}$.
  Then it is easy to see from the assumption that
  \begin{equation} \zeta ({\partial_\pm M}) = t_\pm, \text{ and } | \nabla \zeta (x) | \leq   1. \label{zeta property}
  \end{equation}
Let \( v \) be the spacetime harmonic function given in Proposition
\ref{spectral hkk f} with \( f = - \tfrac{2 (3 - \gamma)}{3 (4 - \gamma)} \eta
\circ \zeta \). A direct calculation with this choice of \( f \) in
Proposition \ref{spectral hkk f} leads to
\begin{align}
  & 2 \int_{\partial_- M} | \nabla v| (-\eta \circ \zeta - (H_{\partial_- M} +
  \gamma u^{- 1} u_{\nu_-}))  \label{neg boundary line} \\
  & \qquad - 2 \int_{\partial_+ M} | \nabla v| (-\eta \circ \zeta + 
  (H_{\partial_+ M} + \gamma u^{- 1} u_{\nu_+})) \label{pos boundary line}
  \\
  \geq    & (6 - 2 \gamma) \int_M \left| \nabla | \nabla v|^{\tfrac{1}{2}} -
  \tfrac{1}{4 - \gamma} \eta \circ \zeta | \nabla v|^{- \tfrac{1}{2}} \nabla v
  \right|^2 - 4 \pi \int_{c_-}^{c_+} \chi (\Sigma_t) \mathrm{d} t \\
  & \quad + \int_M 2 \left( \tfrac{3 - \gamma}{4 - \gamma} (\eta \circ
  \zeta)^2 + \langle \nabla (\eta \circ \zeta), \tfrac{\nabla v}{| \nabla v|}
  \rangle + (- \gamma u^{- 1} \Delta_g u + \tfrac{1}{2} R_g) \right) | \nabla
  v| . \label{ode line} 
\end{align}
We estimate the lines \eqref{neg boundary line}, \eqref{pos boundary line} and
\eqref{ode line}. First, we estimate the integrands in the lines \eqref{neg
boundary line} and \eqref{pos boundary line}. Using the assumptions and \eqref{zeta property}, we see
that
\begin{equation}
  \gamma u^{- 1} u_{\nu_+} + H_{\partial_+ M} \geq    \eta \circ \zeta
  \text{ on } \partial_+ M \text{ and } \gamma u^{- 1} u_{\nu_-} +
  H_{\partial_- M} \geq    -\eta \circ \zeta \text{ on } \partial_- M.
  \label{mc by model}
\end{equation}
Next, by the chain rule and that \( | \nabla \zeta | \leq    1 \),
\[ \langle \nabla (\eta \circ \zeta), \tfrac{\nabla v}{| \nabla v|} \rangle =
   \eta' \circ \zeta \langle \nabla \zeta, \tfrac{\nabla v}{| \nabla v|}
   \rangle \geq    \eta' \circ \zeta ; \]
and \( - \gamma u^{- 1} \Delta_g u + \tfrac{1}{2} R_g \geq    \Lambda \), so
the integrand of the line \eqref{ode line} is greater than \( 2 ((\eta \circ
\zeta)^2 + \eta' \circ \zeta + \Lambda) |\nabla v| \) which vanishes by \eqref{eq general
ode}.
Lastly by the toroidal structure, \( \chi (\Sigma_t) \leq    0 \).
Therefore, we can conclude that
\[ \nabla | \nabla v|^{\tfrac{1}{2}} - \tfrac{1}{4 - \gamma} \eta \circ \zeta
   | \nabla v|^{- \tfrac{1}{2}} \nabla v = 0 \]
almost everywhere, equivalently,
\[ \nabla | \nabla v| - \tfrac{2}{4 - \gamma} \eta \circ \zeta | \nabla v|
   \nabla v = 0. \]
  In the case that equality holds such that the width is
  $t_+ - t_-$, the metric can be split into $g = d t^2 + \phi (t)^2
  g_{\mathbb{T}^2}$, where $d t = | \nabla v |^{- 1} d v$ and $\phi$ satisfies
  that
  \[ 2 \frac{\phi'}{\phi} = \frac{\Delta v - \nabla_{t t} v}{| \nabla v |} = 
     \frac{4 - 2 \gamma}{4 - \gamma} \eta, \]
  therefore we have $\phi (t) = \exp \left( \frac{2 - \gamma}{4 - \gamma}
  \int^t \eta \right)$.
  
  Tracing back all the equalities in Proposition \ref{spectral hkk f}, we see
  from \eqref{determineu} that \( u \) must be a constant multiple of \( | \nabla v|^{\tfrac{1}{2}} \),
  and by a direct calculation, 
it is a constant multiple of $\exp \left( \frac{1}{4 - \gamma} \int^t
\eta \right)$.
\end{proof}

\section{Spectral positive mass theorem}\label{sec pmt}

In this section, we prove our main result the spectral positive mass theorem (Theorem \ref{spec pmt}).
%
%
The proof of the spectral positive mass theorem deals with the same spacetime harmonic function as in the proof of
Theorem \ref{spectral acg}, see Remark \ref{same metric}.

Let \( T_\rho \) be the constant radial coordinate torus in the asymptotic end, and set \( M_\rho \) to be the bounded component of \( M\backslash T_\rho \).

We start with a solution of the spacetime harmonic function \( w  \) which solves
\begin{equation} \label{w}
  \Delta_g w + 3f |\nabla w| =0 \text{ in } M_\rho\text{, }w =0 \text{ on } \partial_-M_\rho, 
  \; w =1 \text{ on } \partial_+ M_\rho, \end{equation}
where we set \( f = - \tfrac{2(3-\gamma)}{3(2-\gamma)}.  \)
Let \( v_\rho = \rho^{\frac{2}{2-\gamma}} w \).

\begin{lemma}
  \label{apriori bound}There exists constants $C > 0$ and $r_{\ast} > 1$, such
  that for all $\rho > r_{\ast}$,
  \[ |v_{\rho} - r^{\tfrac{2}{2 - \gamma}} | \leq    C \text{ on } M_{\rho}
     \backslash M_{r_{\ast}} . \]
\end{lemma}

\begin{proof}
  Let $M_\rho = [r_0, \rho] \times \mathbb{T}^2$ for any $\rho>r_0$, $\partial_+ M_\rho =
  \{\rho\} \times \mathbb{T}^2$ and $\partial_- M_\rho = \{r_0 \} \times
  \mathbb{T}^2$. For any $\rho > r_0$, we denote $w_{\rho }$ the spacetime
  harmonic function with $w_\rho = 0$ on $\partial_- M_\rho$ and $w_\rho = 1$ on
  $\partial_+ M_\rho$. 
  
  Let $r_1\in(r_0,+\infty)$ be determined later. By the Hopf lemma,
  \begin{equation}
    \tfrac{\partial w_{r_1}}{\partial \nu_{r_1}} > 0 \text{ along } \partial_+
    M_{r_1} . \label{normal derivative wr1}
  \end{equation}
  Define
  \begin{equation}
    z^+ = \left\{\begin{array}{ll}
      c_1 w_{r_1} & \text{ on } M_{r_1},\\
      r^{\tfrac{2}{2 - \gamma}} + (c_1 - r_1^{\frac{2}{2-\gamma}} - \lambda r_1^{- 2}) + \lambda
      r^{- 2} & \text{ on } M\backslash M_{r_1} .
    \end{array}\right. \label{z plus}
  \end{equation}
  We show that $z^+$ is a super solution on $M\backslash M_{r_1}$ if $\lambda$
  and $r_1$ are chosen appropriately.
  
  By the asymptotics \eqref{asymptotics g} of $g$,
  \begin{equation}
    \det g = r^2 (1 + r^{- \kappa} \ensuremath{\operatorname{tr}}_g m + o
    (r^{- \kappa})), \label{full det}
  \end{equation}
  and
  \begin{equation}
    g^{r r} = r^2 (1 + o (r^{- \kappa})) . \label{inverse metric rr}
  \end{equation}
  So for all sufficiently large $r > r_1$,
\begin{align}
\Delta z^+ = & \tfrac{1}{\sqrt{\det g}} \partial_r (g^{r r} \sqrt{\det g} \partial_r
z^+) \\
= & \tfrac{2}{2 - \gamma} r^{\tfrac{2}{2 - \gamma}} (\kappa -
\tfrac{\kappa}{2} r^{- \kappa} \mbox{tr}_g m + o
(r^{- \kappa})) \label{laplace z}
\end{align}
  by a tedious calculation. Moreover,
  \[ | \nabla z^+ |^2 = g^{r r} (\partial_r z^+)^2 = r^{\tfrac{4}{2 - \gamma}}
     (\tfrac{4}{(2 - \gamma)^2} - 2 \lambda r^{- \kappa} + o (r^{- \kappa})),
  \]
  and hence
  \begin{equation}
    | \nabla z^+ | = r^{\tfrac{2}{2 - \gamma}} (\tfrac{2}{2 - \gamma} -
    \lambda \tfrac{2 - \gamma}{2} r^{- \kappa} + o (r^{- \kappa})) .
    \label{grad z length}
  \end{equation}
  It follows from \eqref{laplace z} and \eqref{grad z length} that
  \begin{equation}
  \Delta z^+ - \kappa | \nabla z^+ | = \kappa \tfrac{2 - \gamma}{2} r^{- 2}
     (\lambda - \tfrac{2}{(2 - \gamma)^2} \mbox{tr}_g m +
     o (1))< 0 \label{ss prop outside}
     \end{equation}
  on $M\backslash M_{r_1}$ if we choose $\lambda <
  \inf_{\mathbb{T}^2} \tfrac{2}{(2 - \gamma)^2}
  \mbox{tr}_g m$ and $r_1$ sufficiently large.
  
  Now we show that $z^+$ is a weak super solution on $M$ if $c_1$ is chosen
  appropriately. Indeed, first observe that by the choice of $r_1$ above,
  $z^+$ is a super solution on $M\backslash M_{r_1}$; on $M_{r_1}$, $z^+ =
  w_{r_1}$ is spacetime harmonic.
  
  With \eqref{normal derivative wr1}, we choose $c_1$ such that
  \begin{align}
      c_1 \tfrac{\partial w_{r_1}}{\partial \nu_{r_1}} \geq &  
     \tfrac{\partial}{\partial \nu_{r_1}} (r^{\tfrac{2}{2 - \gamma}} + (c_1 -
     r_1 - \lambda r_1^{- 2}) + \lambda r^{- 2})\\ = &r^{\tfrac{2}{2 - \gamma}}
     (\tfrac{2}{2 - \gamma} - 2 \lambda r^{- \kappa} + o (r^{- \kappa}))\label{bdry deri comparison}
  \end{align}
  holds along $T_{r_1}$. The following steps the lines are almost identical to
  {\cite[(6.17)-(6.21)]{alaee-positive-2022}}. For the sake of completeness, we present it here with our notation.
  
  It is easy to see that the spacetime harmonic function $v_\rho$ satisfies the following boundary condition
  \begin{equation}
      \left\{\begin{array}{ll}
           v_\rho=\rho^{\frac{2}{2-\gamma}} & \mbox{on}\;\partial_+ M_\rho \\
           v_\rho=0 & \mbox{on}\;\partial_-M_\rho
      \end{array}\right.
  \end{equation}

  From the definition of $z^+$ and \eqref{ss prop outside}, the function $z^+-v_\rho$ is a super solution for a linear elliptic equation with bounded coefficients as follows
  \begin{equation}
      \mathcal L(z^+- v_\rho):=\Delta(z^+- v_\rho)+\kappa\frac{\nabla(z^++ v_\rho)}{|\nabla z^+|+|\nabla  v_\rho|}\nabla(z^+- v_\rho)\leq 0
  \end{equation}
  holds both on $M_{r_1}$ and $M_\rho/M_{r_1}$. We denote $\overrightarrow{\mathcal G}=\kappa\frac{\nabla(z^++v_\rho)}{|\nabla z^+|+|\nabla v_\rho|}$ Then for any non-negative test function $\phi\in C_c^{\infty}(M_{r_1})$, we have 
  \begin{align}
      0\leq &-\int_{M_{r_1}}\phi\mathcal L(z^+-v_{r_1})dV\\
      = &\int_{M_{r_1}}(\nabla \phi\cdot\nabla(z^+-v_\rho)-\phi \overrightarrow{\mathcal G}\cdot\nabla(z^+-v_\rho))dV\\
      &-\int_{T_{r_1}}\phi\frac{\partial}{\partial\nu_{r_1}}(c_1\omega_{r_1}- v_\rho)dA\label{LHS week sub harmonic}
  \end{align}
  and also
  \begin{align}
      0\leq &-\int_{M_\rho\setminus M_{r_1}}\phi\mathcal L(z^+-v_{r_1})dV\\
      = &\int_{M_\rho\setminus M_{r_1}}(\nabla \phi\cdot\nabla(z^+-v_\rho)-\phi \overrightarrow{\mathcal G}\cdot\nabla(z^+-v_\rho))dV\\
      &+\int_{T_{r_1}}\phi\frac{\partial}{\partial\nu_{r_1}}(r^{\frac{2}{2-\gamma}}+\lambda r^{-2}-v_\rho)dA\label{RHS week subharmonic}
  \end{align}
  By summing both sides of inequalities \eqref{LHS week sub harmonic} and \eqref{RHS week subharmonic}, in the spirit of \eqref{bdry deri comparison}, we obtain
  \begin{align}
      &\int_{M_\rho}(\nabla \phi\cdot\nabla(z^+-v_\rho)-\phi \overrightarrow{\mathcal G}\cdot\nabla(z^+- v_\rho))dV\\\geq &\int_{T_{r_1}}\phi\frac{\partial}{\partial\nu_{r_1}}(c_1w_{r_1}-r^{\frac{2}{2-\gamma}}-\lambda r^{-2})dA\\
      \geq & 0\label{weak sup harmonic}
  \end{align}
  
  We derive from the weak maximum principle that
  \[ \inf_{M_{r_1}} (z^+ - v_{r_1}) \geq    \inf_{\partial M_{r_1}} (z^+
     - v_{r_1}) \geq    0. \]
  On the other hand, let $r_2>1$. The construction of a lower barrier $z^-$ is analogous, which is given by
  \begin{equation*}
    z^- = \left\{\begin{array}{ll}
      c_2 \tilde w_{r_2} & \text{ on } M_{r_2},\\
      r^{\tfrac{2}{2 - \gamma}} + (c_2 - r_2^{\frac{2}{2-\gamma}} - \chi r_2^{- 2}) + \chi
      r^{- 2} & \text{ on } M\backslash M_{r_2} .
    \end{array}\right.
  \end{equation*}
  where $\tilde\omega_{r_2}$ is a spacetime harmonic function satisfying that $\tilde\omega_{r_2}=-1$ on $\partial_+M_{r_2}$ and $\tilde\omega_{r_2}=0$ on $\partial_-M_{r_2}$, and similarly $\chi$ and $r_2$ are chosen such that $z^-$ is a sub solution for the spacetime harmonic equation. The same comparison argument proves that $v_\rho \geq z^-$ on $M_\rho$ for any $\rho>r_2$. Therefore, by choosing $r_*=\max\{r_0,r_1\}$, we have $z^-\leq v_\rho\leq z^+$ on $M\rho$ for any $\rho>r_*$.  Hence, the lemma is proved.
\end{proof}

\begin{lemma}\label{c1}
  There exists constants $C > 0$ and $r_{\ast} > 1$, such that for all $\rho >
  r_{\ast}$,
  \[ | \nabla v_{\rho} - \nabla r^{\tfrac{2}{2 - \gamma}} | \leq    C \text{
     on } M_{\rho} \backslash M_{r_{\ast}} . \]
\end{lemma}

\begin{proof}
The proof is basically the same with {\cite[Lemma 6.1]{alaee-positive-2022}}
  with appropriate adjustments similar to those of Lemma \ref{apriori bound}.
  
Given $\rho>r_*$, let $h_\rho=v_\rho-r^{\frac{2}{2-\gamma}}$. Note that $h_\rho$ satisfies the following equation
\begin{equation}
    \Delta h_\rho+\kappa\frac{\nabla(v_\rho+r)}{|\nabla v_\rho|+|\nabla r^{\frac{2}{2-\gamma}}|}\nabla{h_\rho}=-\Delta r^\frac{2}{2-\gamma}-\kappa|\nabla r^{\frac{2}{2-\gamma}}|:=G.\label{equation for gradient estimate}
\end{equation}
It is straightforward to verify that the coefficients of the first-order terms in \eqref{equation for gradient estimate} are uniformly bounded. Fix $p_0\in\partial_+M_\rho$, and denote by $B_\epsilon$ (resp. $B_{\epsilon/2}$) the geodesic ball centered at $p_0$ of radius $\epsilon$ (resp. $\epsilon/2$). We fix $\epsilon>0$ so that it is smaller than the injectivity radius at every point $x\in M\setminus M_{r_*}$. For $1<p<\infty$, the boundary $L^p$-estimates, together with the condition $h_\rho=0$ on $\partial_1^+M_\rho$, then yield
\begin{equation}
    \|h_\rho\|_{W^{2,p}(B_{\epsilon/2}\cap M_\rho)}\leq C_0(\|G\|_{L^p(B_{\epsilon}\cap M_\rho)}+\|h_\rho\|_{L^p(B_\epsilon\cap M_\rho)})
\end{equation}
Since the metric is asymptotically locally hyperbolic, the constant $C_0$ is uniform over all $x_0\in\partial_1^+M_\rho$ and all $\rho>r_*$. On the other hand, a direct calculation gives
\begin{equation}
    \Delta r^{\frac{2}{2-\gamma}}+\kappa|\nabla r^{\frac{2}{2-\gamma}}|=- \tfrac{\kappa}{(2 - \gamma)} \mbox{tr}_g m r^{-2} + o (1).
\end{equation}
Therefore, $G$ is uniformly bounded. Together with the last lemma, this implies that $h_\rho$ is also uniformly bounded on $M_\rho\setminus M_1$, independently of $\rho$. By choosing $p>3$ and applying the Sobolev embedding theorem, there exists a uniform constant $C$ such that
\begin{equation}
    \|h_\rho\|_{C^{1,1-\frac{3}{p}}(B_{\epsilon/2}\cap M_\rho)}\leq C_1\|v_\rho\|_{W^{2,p}(B_{\epsilon/2}\cap M_\rho)}\leq C.
\end{equation}
Interior $L^p$-estimates can be used to obtain the same conclusion for balls away from the boundary. The desired result follows.
\end{proof}


We derive the asymptotics of the spectral mean curvature in the following lemma. 

\begin{lemma} \label{spec mc decay} The spectral mean curvature of the coordinate torus \( T_r \) satisfies the following
\begin{equation} H + \gamma u^{- 1} u_{\nu} = (2 + \tfrac{\gamma}{2 - \gamma}) - \kappa
   (\tfrac{1}{2} \ensuremath{\operatorname{tr}}_{g_{\mathbb{T}^{2}}} m + \gamma \zeta) r^{-
   \kappa} + o (r^{- \kappa}) . \end{equation}

  \end{lemma}
 
\begin{proof}
The asymptotics \eqref{asymptotics g} is written in terms of components of $g$
as $g_{r r} = r^{- 2} (1 + o (r^{- \kappa}))$, $g_{r i} = o (r^{- \kappa})$
and $g_{i j} = r^2 (\delta_{i j} + r^{- \kappa} m_{i j} + o (r^{- \kappa}))$.
Then $g^{r r} = r^2 (1 + o (r^{- \kappa}))$, $g^{r i} = o (r^{- \kappa})$.
%
The unit normal $\nu$ of $r$-level set $T_r$ is given by
\begin{equation}
  \nu = (g^{r r})^{- \tfrac{1}{2}} g^{r \alpha} \partial_{\alpha}
  \label{normal}
\end{equation}
where $\alpha$ ranges from $1$ to $3$. The mean curvature $H$ of $T_r$ is
\begin{equation}
  H = \sigma^{i j} \langle \nu, \nabla_{\partial_i} \partial_j \rangle = -
  \sigma^{i j} (g^{r r})^{- \tfrac{1}{2}} \Gamma_{i j}^r, \label{mean
  curvature}
\end{equation}
where $\sigma_{i j} = g_{i j}$ is the the induced metric on $T_r$ and
$\sigma^{i j}$ is its inverse.

We compute componentwise and write the quantities in the form of an asymptotic expansion.

First of all, 

\begin{equation}
    (g^{r r})^{- \tfrac{1}{2}} = r^{- 1} (1 + o (r^{- \kappa})).\label{grr}
\end{equation}

Secondly, in terms of $\sigma$, we have
\begin{equation}
    \det \sigma = r^4 (1 + r^{- \kappa} \mbox{tr}_{g_{\mathbb{T}^{2}}}m
+ o (r^{- \kappa})),\label{det}
\end{equation}
and

\begin{equation}
    \sigma^{- 1} = (\det \sigma)^{- 1} r^2 \left(\begin{array}{cc}
     1 + r^{- \kappa} m_{22} + o (r^{- \kappa}) & - r^{- \kappa} m_{12} + o
     (r^{- \kappa})\\
     - r^{- \kappa} m_{12} + o (r^{- \kappa}) & 1 + r^{- \kappa} m_{11} + o
     (r^{- \kappa})
   \end{array}\right).\label{sigmainverse}
\end{equation}

Set $\hat{m} = \left(\begin{array}{cc}
  m_{22} & - m_{1 2}\\
  - m_{1 2} & m_{11}
\end{array}\right)$,
we have 
\begin{align}
\sigma^{i j} = & (\det \sigma)^{- 1} r^2 (\delta_{i j} + r^{- \kappa}
\hat{m}_{i j} + o (r^{- \kappa})) \\
= & r^{- 2} (1 - r^{- \kappa} \mbox{tr}_{g_{\mathbb{T}^{2}}} m + o (r^{-
\kappa})) (\delta_{i j} + r^{- \kappa} \hat{m}_{i j} + o (r^{- \kappa} ))
.\label{sigmainverse2}
\end{align}
Then we compute the Christoffel symbols of \(g \)
\begin{align}
\Gamma_{i j}^r = & \tfrac{1}{2} g^{r r} (g_{i r, j} + g_{j r, i} - g_{i j, r}) +
\tfrac{1}{2} g^{r k} (g_{i k, j} + g_{j k, i} - g_{i j, k}) \\
= & - \tfrac{1}{2} g^{r r} g_{i j, r} + o (r^{2 - 2 \kappa}) \\
= & - \tfrac{1}{2} r^2 (1 + o (r^{- \kappa})) \tfrac{\partial}{\partial r}
(r^2 (\delta_{i j} + r^{- \kappa} m_{i j} + o (r^{- \kappa}))) + o (r^{2 - 2
\kappa}) \\
= & - r^3 \delta_{i j} - \tfrac{1}{2} (2 - \kappa) r^{- \kappa + 3} m_{i j}
+ o (r^{- \kappa + 3}) .\label{chsymbol}
\end{align}
By Combining \eqref{det}, \eqref{sigmainverse}, \eqref{sigmainverse2}, \eqref{chsymbol} and \eqref{mean curvature}, we obtain
\begin{align}
H = & - \sigma^{i j} (g^{r r})^{- \tfrac{1}{2}} \Gamma_{i j}^r \\
= & (1 - r^{- \kappa} \mbox{tr}_{g_{\mathbb{T}^{2}}} m+ o (r^{- \kappa}))
(\delta_{i j} + r^{- \kappa} \hat{m}_{i j} {+ o (r^{- \kappa}} ))
\\
& \quad (1 + o (r^{- \kappa})) (- \delta_{i j} - \tfrac{1}{2} (2 - \kappa)
r^{- \kappa} m_{i j} + o (r^{- \kappa})) \\
= & 2 - \tfrac{1}{2} \kappa r^{- \kappa} \mbox{tr}_{g_{\mathbb{T}^{2}}} m
+ o (r^{- \kappa}).\label{mean curvature2}
\end{align}
Next we evaluate the term $u^{- 1} u_{\nu}$. By \eqref{normal}, we have
\[ \nu = r (1 + o (r^{- \kappa})) \partial_r + \sum_i o (r^{- \kappa - 1})
   \partial_i . \]
Therefore, 
\begin{align}
u^{- 1} u_{\nu} & = r^{- \tfrac{1}{2 - \gamma}} (1 + r^{- \kappa} \zeta + o
(r^{- \kappa}))^{- 1} \\
& \quad (r (1 + o (r^{- \kappa})) \partial_r + \sum_i o (r^{- \kappa - 1})
\partial_i) (r^{\tfrac{1}{2 - \gamma}} (1 + r^{- \kappa} \zeta + o (r^{-
\kappa}))) \\
& = r^{- \tfrac{1}{2 - \gamma}} (1 - r^{- \kappa} \zeta + o (r^{- \kappa}))
r \partial_r (r^{\tfrac{1}{2 - \gamma}} (1 + r^{- \kappa} \zeta)) + o (r^{-
\kappa}) \\
& = \tfrac{1}{2 - \gamma} - \kappa \zeta r^{- \kappa} + o (r^{- \kappa}) .
\end{align}
Hence the spectral mean curvature $H + \gamma u^{- 1} u_{\nu}$ is given by
\[ H + \gamma u^{- 1} u_{\nu} = (2 + \tfrac{\gamma}{2 - \gamma}) - \kappa
   (\tfrac{1}{2} \ensuremath{\operatorname{tr}}_{g_{\mathbb{T}^{2}}} m + \gamma \zeta) r^{-
   \kappa} + o (r^{- \kappa}) . \]
This completes the proof of the lemma.
\end{proof}

With the PDE estimates and the asymptotics of the spectral mean curvature, we can finally prove our main result Theorem \ref{spec pmt}.

\begin{proof}[Proof of Theorem \ref{spec pmt}]
  We apply Proposition \ref{spectral hkk f} to \( v_\rho \), and we obtain the following integral inequality
\begin{align}
& 2 \int_{\partial_- M_\rho} | \nabla v_\rho| \left( - { \tfrac{4 -
\gamma}{2 - \gamma}} - (H_{\partial_- M_{r}} + \gamma u^{- 1} u_{\nu_-})
\right) \label{pmt l1}  \\
& \qquad - 2 \int_{\partial_+ M_\rho} | \nabla v_\rho| \left( -
{ \tfrac{4 - \gamma}{2 - \gamma}} + (H_{\partial_+ M_\rho} + \gamma u^{- 1}
u_{\nu_+}) \right) \label{pmt l2} \\
\geq    & (6 - 2 \gamma) \int_{M_\rho} \left| \nabla | \nabla v_\rho|^{\tfrac{1}{2}}
- {\tfrac{1}{ 2 - \gamma}} | \nabla v_{r}|^{-
\tfrac{1}{2}} \nabla v_\rho \right|^2 - \int_{c_-}^{c_+} 4 \pi \chi (\Sigma_t)
\mathrm{d} t \label{pmt l3}\\
& \quad + \int_{M_\rho} 2 \left(- \gamma u^{- 1} \Delta_g u + \tfrac{1}{2} R_g +\tfrac{(3-\gamma)(4-\gamma)}{(2-\gamma)^2}
  \right) | \nabla v_\rho| \label{pmt l4}.
\end{align}
By the assumptions of the theorem,
\begin{equation} \label{bdry integral}
 \qquad - 2 \int_{\partial_+ M_\rho} | \nabla v_\rho| \left( -
{ \tfrac{4 - \gamma}{2 - \gamma}} + (H_{\partial_+ M_\rho} + \gamma u^{- 1}
u_{\nu_+}) \right) \geq    0.
\end{equation}

To show the spectral positive mass theorem, it suffices to show that the above converges to the spectral mass \( E= E(M,g,\gamma,u) \) as \( \rho\to \infty \).

Using Lemma \ref{apriori bound}, \( v_\rho \) is locally uniformly bounded. Standard elliptic estimates yield locally uniform \( C^{2,\alpha} \) bounds for \( v_\rho \) for any \( \alpha \in (0,1) \). Then we can use the Arzelà-Ascoli lemma to extract a convergent subsequence \( v_{\rho_i} \). We denote the limit by \( v \).

It follows from Lemma \ref{c1} that
\[
  |\nabla v_\rho | = |\nabla r ^{\tfrac{2}{2-\gamma}} |\big|_{r=\rho} + O(1) = \tfrac{2}{2-\gamma} \rho^{\tfrac{2}{2-\gamma}} + O(1).
\]
Note that \( T_\rho = \partial_+ M_\rho \). Hence the boundary integral \eqref{bdry integral} is
\begin{align}
 0 \leq    & - 2 \int_{\partial_+ M_\rho} | \nabla v_\rho | \left( - \tfrac{4 - \gamma}{2 -
  \gamma} + (H_{\partial_+ M_\rho} + \gamma u^{- 1} u_{\nu_+}) \right)
  \\
  = & 2 \int_{\mathbb{T}^2} (\tfrac{2}{2 - \gamma} \rho^{\tfrac{2}{2 - \gamma}} +
  O (1)) \kappa (\tfrac{1}{2}
  \ensuremath{\operatorname{tr}}_{g_{\mathbb{T}^2}} m + \gamma \zeta) \rho^{-
  \kappa} (\rho^2 + o (\rho^2)) \sqrt{g_{\mathbb{T}^2}} \\
  = & \tfrac{4}{2 - \gamma} \kappa \int_{\mathbb{T}^2} (\tfrac{1}{2}
  \ensuremath{\operatorname{tr}}_{g_{\mathbb{T}^2}} m + \gamma \zeta) \sqrt{g_{\mathbb{T}^{2}}} + o (1) = E+o(1) \label{mass convergence}
  , 
\end{align}
where we have used Lemma \ref{spec mc decay}.
Hence, we have shown the spectral positive mass theorem.

It remains to show the case of vanishing spectral mass.
To this end, we need to take limits of both sides of the inequality in the lines \eqref{pmt l1}-\eqref{pmt l4}.
We only have to deal with the convergence of
\[
  \int_{M_{\rho_i}} \left| \nabla | \nabla v_{\rho_{i}}|^{\tfrac{1}{2}}
- {\tfrac{1}{ 2 - \gamma}} | \nabla v_{r_{i}}|^{-
  \tfrac{1}{2}} \nabla v_{\rho_{i}} \right|^2,
\]
since the convergence of other terms follows simply from \( C^{2,\alpha} \) convergence. The issue is due to the points where \( \nabla v_{r_{i}} \)
vanishes.

To remedy this, we fix a compact subset \( \Omega \)
of \( M \), and define
\[ \Omega_{\epsilon} = \{ x\in \Omega:\text{ } |\nabla v| \geq    \epsilon \} \]
The convergence
\[
\left| \nabla | \nabla v_{r_{i}}|^{\tfrac{1}{2}}
- {\tfrac{1}{ 2 - \gamma}} | \nabla v_{r_{i}}|^{-
  \tfrac{1}{2}} \nabla v_{r_{i}} \right|^2
\to
\left| \nabla | \nabla v|^{\tfrac{1}{2}}
- {\tfrac{1}{ 2 - \gamma}} | \nabla v|^{-
  \tfrac{1}{2}} \nabla v \right|^2 \text{ on }\Omega_{\epsilon}
\]
follows again from the \( C^{2,\alpha} \) convergence of \( v_{\rho_i}  \) to \( v \).
And
it follows from Fatou's lemma that
\begin{align}
  & \liminf_{i \to \infty} \int_{\Omega} | \nabla | \nabla v_{\rho_i}
  |^{\tfrac{1}{2}} - \tfrac{1}{2 - \gamma} | \nabla v_{\rho_i} |^{- \tfrac{1}{2}}
  \nabla v_{\rho_i} |^2 \\
  \geq    & \liminf_{i \to \infty} \int_{\Omega_{\varepsilon}} | \nabla |
  \nabla v_{\rho_i} |^{\tfrac{1}{2}} - \tfrac{1}{2 - \gamma} | \nabla v_{r_i}
  |^{- \tfrac{1}{2}} \nabla v_{\rho_i} |^2 \\
  \geq    & \int_{\Omega_{\varepsilon}} | \nabla |
  \nabla v|^{\tfrac{1}{2}} - \tfrac{1}{2 - \gamma} | \nabla v|^{-
  \tfrac{1}{2}} \nabla v |^2 . 
\end{align}
Hence by letting $\epsilon\rightarrow 0$,
\begin{align}
  & \liminf_{i \to \infty} \int_{\Omega} | \nabla | \nabla v_{\rho_i}
  |^{\tfrac{1}{2}} - \tfrac{1}{2 - \gamma} | \nabla v_{\rho_i} |^{- \tfrac{1}{2}}
  \nabla v_{\rho_i} |^2 \\
  \geq    &  \int_{\Omega} | \nabla |
  \nabla v|^{\tfrac{1}{2}} - \tfrac{1}{2 - \gamma} | \nabla v|^{-
  \tfrac{1}{2}} \nabla v |^2 . 
\end{align}
Considering the above, the convergence \eqref{mass convergence} and \( \chi(\Sigma_t) = 0 \)
in the lines \eqref{pmt l1}-\eqref{pmt l4},
we obtain the positive mass inequality,
\begin{align}
E
\geq    & (6 - 2 \gamma) \int_{M} \left| \nabla | \nabla v|^{\tfrac{1}{2}}
- {\tfrac{1}{ 2 - \gamma}} | \nabla v|^{-
\tfrac{1}{2}} \nabla v \right|^2 \mathrm{d} t \\
& \quad + \int_{M} 2 (- \gamma u^{- 1} \Delta_g u + \tfrac{1}{2} R_g +\tfrac{(3-\gamma)(4-\gamma)}{(2-\gamma)^2}
  ) | \nabla v| \\
  & \quad + 2 \int_{\partial_- M} | \nabla v| \left(  { \tfrac{4 -
\gamma}{2 - \gamma}} + (H_{\partial_- M} + \gamma u^{- 1} u_{\nu_-})
\right) .
\end{align}
Vanishing mass then implies that
\[  \nabla | \nabla v|^{\tfrac{1}{2}}
- {\tfrac{1}{ 2 - \gamma}} | \nabla v|^{-
  \tfrac{1}{2}} \nabla v = 0 \]
holds almost everywhere. The rest of the proof is similar to that of
Theorem \ref{spectral acg} (see Remark \ref{same metric}). Therefore, if the mass $m=0$, $M$ is isometric to a hyperbolic cusp.
\end{proof}

\bibliographystyle{alpha}
\bibliography{cusp-pmt}

\end{document}